\documentclass[11pt,a4paper]{article}

\usepackage[latin1]{inputenc}
\usepackage{amssymb, amscd, amsmath, amsfonts}
\usepackage{amsthm}
\usepackage{mathrsfs}
\usepackage{microtype}

\newtheorem{theorem}{Theorem}
\newtheorem{lemma}{Lemma}
\newtheorem{claim}{Claim}
\newtheorem{corollary}[theorem]{Corollary}
\newtheorem{proposition}[theorem]{Proposition}
\newtheorem{remark}{Remark}

\newcommand{\Prob}{\mathbb P\,}
\newcommand{\Esp}{\mathbb E\,}
\newcommand{\N}{\mathbb N}
\newcommand{\R}{\mathbb R}
\newcommand{\Jc}{\textrm{J}}
\newcommand{\X}{\textrm{X}}
\newcommand{\Zc}{\textrm{Z}}
\newcommand{\Ic}{\mathcal I}
\newcommand{\indicator}{\boldsymbol 1}
\begin{title}
{Rice Formula for processes with jumps and applications}
\end{title}
\author{Federico Dalmao\thanks{
Mailing address: Departamento de Matem\'{a}tica y Estad\'{i}stica del Litoral, Regional Norte,  
Universidad de la Rep\'{u}blica, Rivera 1350, 50000, Salto, Uruguay. Telephone number: 598 4734 2924. 
E-mail: fdalmao@unorte.edu.uy.} 
\qquad 
Ernesto Mordecki\thanks{
Mailing address: Centro de Matem\'atica, Facultad de Ciencias, Universidad de la Rep\'{u}blica, Igu\'a 4225, 11400, Montevideo, Uruguay.
Telephone number: 598 2525 2522,
E-mail: mordecki@cmat.edu.uy,
Fax number: 598 2522 0653.}
}
\begin{document}
\maketitle
\begin{abstract}
We extend Rice Formula to a process which is the sum of two independent processes: 
a smooth process and a pure jump process with finitely many jumps. 
Formulas for the mean number of both continuous and discontinuous crossings 
through a fixed level on a compact time interval are obtained. 
We present examples in which we compute explicitly the mean number of crossings and 
compare which kind of crossings dominate for high levels. 
In one of the examples the leading term of the tail of the distribution function 
of the maximum of the process over a compact time interval as the level goes to infinity is obtained. 
We end giving a generalization, to the non-stationary case, of Borovkov-Last's Rice Formula for 
Piecewise Deterministic Markov Processes.
\end{abstract}

\section{Introduction}\label{sec:intro}
In the present work, we are interested in the number of crossings through a fixed level by 
a class of stochastic processes having finitely many jumps on compact intervals  
and smooth stochastic evolution between the jumps. 
More precisely, we consider a process $\X$ which can be written in the form $\X=\Zc+\Jc$  
where $\Zc$ is a process with continuously differentiable paths and $\Jc$ is a pure jump process, 
independent from $\Zc$. 
Thus, $\Zc$ describes the continuous part of $\X$ and $\Jc$ describes the jumps of $\X$.

Such a process can cross a fixed level $u$ at a continuity point or at a jump one. 
Under general conditions, 
we obtain Rice-type formulas for the mean number of (both) continuous and discontinuous crossings through a fixed level $u\in\R$ 
on a compact time interval.
Afterwards, we give explicit algebraic expressions for the mean number of crossings 
in two examples. Then, 
we compare which kind of crossing (continuous/discontinuous) 
dominates as the level $u$ goes to infinity and find different behaviors on these examples. 
Besides, in the first example we 
derive second order Rice formulas for
the continuous and discontinuous crossings and use them 
to obtain the exact asymptotic leading term  
of $\Prob(M(T)>u)$ as $u\to\infty$ where   
$M(T)=\max\{X(t):t\in[0,T]\}$.

The paper is organized as follows.
Section 2 gives some background on Rice formulas 
and compares the present work with similar previous ones. 
Section 3 starts with the 
definitions and then
presents the main result (Theorem \ref{teo:nogauss}). 
Section 4 presents the examples. 
Section 5 moves a bit from the main line 
presenting a generalization of Borovkov-Last's Rice-type formula to the non-stationary case. 
Finally, our main result is proved in Section 6.

This work began under the initiative of Mario Wschebor, and is largely inspired by his fundamental
contribution in the field, masterly exposed in \cite{aw}.

\section{Background}
Counting the number of crossings through a fixed level by a stochastic process 
is a classical problem in Probability Theory. 
Nevertheless little is known, in general, about the distribution of the random variable \emph{number of crossings}. 
Therefore, Rice Formulas, which give expressions for the moments of this r.v, are very important and useful. 

Rice Formulas are named after S.O. Rice who 
obtained  by informal arguments a simple formula for the mean number of crossings through a fixed level 
for stationary Gaussian processes in 1944 \cite{Rice44,Rice45}. 
This formula was afterwards proved under successively weaker conditions, see \cite{aw,kratz,Lead} and references therein. 
The major part of the literature on the subject is devoted to the Gaussian case, in particular to the stationary one. 
Extensions for general Gaussian, non-Gaussian processes and fields were stated afterwards, 
see \cite{adler-samo,aw,brillinger,kratz,marcus,zahle}.
Some recent extensions to non-Gaussian processes include 
the Shot Noise processes \cite{shot,shot-saltos}, the Generalized Hyperbolic Process \cite{hyper}, 
the Laplace Moving Average \cite{galtier} 
and the Piecewise Deterministic Markov Processes \cite{bor-last,bor-lastn}.

The applications of Rice Formulas include telecommunications and signal processing \cite{Rice44,Rice45}; 
reliability theory in engineering \cite{Rychlik}; 
oceanography (the height of sea waves) \cite{sea waves,longuett}; 
physics and astronomy: random mechanics \cite{kree},  
the Shot Noise processes \cite{shot,shot-saltos} and microlensing \cite{Micro}; 
random systems of polynomial equations \cite{Mario counting,arw}, 
among others.

The case of discontinuous processes has been treated only recently in the literature.
Borovkov and Last \cite{bor-last,bor-lastn} 
consider the number of continuous crossings of a discontinuous process with random jumps and deterministic evolution between jumps.
The continuous and discontinuous parts of the process are not independent, 
hence, our main result does not apply. 
Nevertheless, we extend the Rice formula given in \cite{bor-last} in Section 5. 
Bierm\'e and Desoulneux \cite{shot-saltos}, 
consider a process obtained 
by superposition of randomly scaled and translated copies 
of a given deterministic discontinuous function (Shot noise process).
They are interested in the expectation of the total number of crossings for different levels for this particular process,
and obtain the Fourier transform of this expectation with respect to these levels. 

\section{Preliminaries and Main Result}\label{section:2}
Let $f:[0,T]\to\R$ be c\`{a}dl\`{a}g and $C^{1}$ between jumps. 
We say that $f$ has a 
\begin{enumerate}
 \item[$\cdot$] continuous crossing through the level $u$ at $s\in(0,T)$ 
if $f$ is continuous at $s$, $f(s)=u$ and 
$f^{\prime}(s)\neq 0$; 
 \item[$\cdot$] a discontinuous crossing if $(f(s^{-})-u)(f(s)-u)<0$.
\end{enumerate}
If $f^{\prime}(s)>0$ in the continuous case, or $f(s^{-})<u<f(s)$ in the discontinuous one, we say that 
$f$ has an up-crossing at $s$, otherwise the crossing is a down-crossing.
\begin{remark}
For the processes considered here this definition coincides with that of 
Scheutzow \cite{scheutzow} and 
in the case of continuous crossings also with that of 
Leadbetter, Lindgren and Rootz\'{e}n \cite{LLR}.
\end{remark}

Now, we introduce the processes, all defined on $[0,T]$. 
Let $\Zc=(Z(t))$ and $\Jc=(J(t))$ be independent stochastic processes 
and define 
$\X=(X(t))$ with $X(t)=Z(t)+J(t)$.

We assume some regularity conditions on the process $\Zc$, see \cite{aw}, namely: 
\begin{enumerate}
\item[$A1:$] the paths of $\Zc$ are $C^1$, almost surely.

\item[$A2:$] the density $p_{Z(t)}(x)$ is jointly continuous for $t\in[0,T]$ and $x$ in a neighborhood of $u$.
Furthermore, assume that for every $t,t^{\prime}\in[0,T]$ the joint distribution of $(Z(t),\dot{Z}(t^\prime))$ 
has a density $p_{Z(t),\dot{Z}(t^{\prime})}(x,x^{\prime})$ 
which is continuous w.r.t. $t$ ($t^{\prime},x,x^{\prime}$ fixed) 
and w.r.t. $x$ at $u$ ($t,t^{\prime},x^{\prime}$ fixed). 

\item[$A3:$] for every $t\in[0,T]$ there exists a continuous version of the conditional expectation 
$\Esp(\dot{Z}(t)\mid X(t)=x)$ for $x$ in a neighborhood of $u$.

\item[$A4:$] the modulus of continuity of $\dot{Z}$ tends to $0$ if $\delta\to 0$:
\begin{equation*}
\mathop{\sup}\limits_{0\leq s<t\leq T,\,|t-s|<\delta}|\dot{Z}(t)-\dot{Z}(s)|\mathop{\to}\limits_{\delta\to 0}0.
\end{equation*}
\end{enumerate}

We also assume that 
\begin{enumerate}
 \item[$B:$] $\Jc$ is a pure jump process with almost surely finitely many jump epochs 
on any compact time interval.
\end{enumerate}
The $n$-th jump epoch and the $n$-th jump magnitude of $\Jc$ are denoted respectively 
by $\tau_n$ and $\xi_n$. 
The process $(\tau_n,\xi_n)_{n\in\N}$ is a Marked Point Process, 
see \cite{Jac} for the general construction based on Markov kernels. 
Finally, for $\tau_n\leq t<\tau_{n+1}$ let $\nu_{t}=n$. Hence
\begin{equation*}
J(t)=\sum^{\nu_t}_{k=0}\xi_{k}.  
\end{equation*}

We can identify the marked point process $\big(\tau_n,J(\tau_n)\big)_{n\in\N}$ 
with its associated random counting measure (RCM) 
$\mu:=\sum^{\nu_T}_{n=0}\delta_{(\tau_n,J(\tau_n))}$ 
on $[0,T]\times\R$, where $\delta_{(t,x)}$ is the Dirac Delta measure concentrated at the point $(t,x)$. 
Note that the RCM $\mu$ of a Borel set $C$ in $[0,T]\times\R$ 
is just the number of points (jump epochs and marks) of the marked point process that lie in $C$.

There exists a random measure $L(dt,dy)$ on $[0,T]\times\R$, 
called the compensating measure of $\mu$, 
such that the processes $t\mapsto L([0,t],A)$ are 
predictable for any Borel set $A$ in $\R$ and, 
under quite general conditions (for example the absolute continuity of the kernels and the finiteness of the intensity of jumps), 
$L$ can be written in terms of ordinary Lebesgue integrals. 
Furthermore
\begin{equation*}
\Esp\int_{[0,T]\times\R}fd\mu=\Esp\int_{[0,T]\times\R}fL(dt,dy)
\end{equation*}
for any predictable function $f$, see \cite{Jac}. 

We can apply the definition of crossing to almost all of the paths. 
Denote $N^{c}_{u}$ (resp. $N^{d}_{u}$) the random number of continuous (resp. discontinuous) 
crossings through the level $u$ by the process $\X$ on the interval $[0,T]$ 
and by $N_{u}$ the total number of crossings. 
For up-crossings we use the same notation with $N$ replaced by $U$. 
It is clear that $N_{u}=N^{c}_{u}+N^{d}_{u}$, and consequently that $\Esp N_{u}=\Esp N^{c}_{u}+\Esp N^{d}_{u}$.
\begin{theorem}\label{teo:nogauss}
Let $\Zc$ and $\Jc$ be two independent processes on $[0,T]$ 
verifying conditions $A1$-$A4$ and $B$ respectively.  
Then, the mean number of continuous and discontinuous crossings through the level $u$ 
by the process $\X=\Zc+\Jc$ on the interval $[0,T]$ are given respectively by:
\begin{multline*}
\Esp N^{c}_{u}=\int_{[0,T]}\Esp(|\dot{X}(t)|\mid X(t)=u)\,p_{X(t)}(u)dt\\
=\int_{[0,T]}\Esp(|\dot{Z}(t)|\mid X(t)=u)\,p_{X(t)}(u)dt,
\end{multline*}
and
\begin{equation*}
\Esp N^{d}_{u}=\Esp\iint_{[0,T]\times\R} \indicator{\{(X(t^{-})-u)(X(t^{-})+y-u)<0\}}L(dt,dy),
\end{equation*}
where $\indicator A$ is the indicator function of the set $A$ and
$L$ is the compensating measure of the random counting measure generated on $[0,T]\times\R$ by the jump process $\Jc$. 
Similar formulas hold for the number of up and down crossings.
\end{theorem}
The proof is postponed to Section 6, let us say here that the 
continuous and the discontinuous crossings are treated separately and by different methods.
An interesting feature of the approach we propose is the fact that Kac Formula \cite{kac} counts only the continuous crossings.
The discontinuous ones are counted using techniques from Point Processes Theory.

Some remarks are in order. 
First, 
note that when the jump process $\Jc$ vanishes, that is, $J(t)=0$ almost surely for all $t\in[0,T]$, 
Theorem \ref{teo:nogauss} reduces to Classical Rice Formula for the process $\Zc$. 

Next, 
observe that the random variable $J(t)$ does not need to have a density for each $t$, but, 
in case it does we have a more explicit result.
\begin{corollary}\label{cor:pzpj}
If $J(t)$ has a continuous density $p_{J(t)}(x)$ for $t\in[0,T]$ and $x\in\R$, we can also write
\begin{equation*}
\Esp N^{c}_{u}=
\int_{[0,T]}dt\int_{\R}\Esp\left(|\dot{Z}(t)|\mid Z(t)=v
\right)p_{Z(t)}(v)p_{J(t)}(u-v)dv. 
\end{equation*}
\end{corollary}

Finally, 
a careful analysis of the proof of Theorem \ref{teo:nogauss} in Section 6 shows that 
the result in Theorem \ref{teo:nogauss} holds true whenever the law of the process $\Zc$, restricted to the subintervals $[\tau_i,\tau_{i+1}]$, 
conditioned to the paths of the jump process verifies the hypothesis $A1$ - $A4$. 
Besides, $A3$ can be weakened assuming that the product of the conditional expectation and the density of $X(t)$ 
(or $Z(t)$ in Corollary \ref{cor:pzpj}) is continuous.

We end this section specializing these results to the case where 
$\Zc$ is a Gaussian process, here, the hypothesis of continuity of the densities and of the 
conditional expectation may be released 
since they follow from the conditions of non-degeneracy of the distribution of $Z(t)$, $t\in[0,T]$, 
and on the regularity of the paths. 
Besides, the ingredients in the formulas are computable explicitly.
\begin{corollary}\label{theorem:gaussdensity}
Let $\Zc$ be a Gaussian process with $C^1$ paths such that 
for every $t$ the distribution of $Z(t)$ is non-degenerated,  
assume further that $\Jc$ is a pure jump process independent from $\Zc$ 
with finite intensity of jumps. 
Then the result of Theorem \ref{teo:nogauss} holds true.
\end{corollary}
\begin{corollary}\label{theorem:gauusest}
Under the hypothesis of Corollary \ref{theorem:gaussdensity}, 
if in addition $\Zc$ is stationary, 
then the formula for the continuous crossings reduces to:
\begin{equation*}
\Esp{N^{c}_{u}}=\Esp{|\dot{Z}(0)|} \int^{T}_{0}p_{X(t)}(u)dt.
\end{equation*}
\end{corollary}

\begin{remark}\label{remark:gauusest}
Let us define the density 
$p(u):=\frac{1}{T}\int^{T}_{0}p_{X(t)}(u)\,dt$. 
Then, we can write
\begin{equation*}
\Esp{N^{c}_{u}}=T\,\Esp{|\dot{Z}(0)|}\,p(u).
\end{equation*}
Observe that $p$ is a mixture of the density of $X(t)$ for $t\in[0,T]$. 
In particular, 
if $\X$ is a stationary process, the density $p$ reduces to that of $X(0)$, just as in the original formula due to Rice.\\
\end{remark}

\section{Examples}
In this section we present two examples. 
In both of them $\Zc$ is a differentiable centered stationary Gaussian process. 
Let $\Gamma$ be the covariance function of $\Zc$, i.e: $\Gamma(\tau)=\Esp Z(0)Z(\tau)$. 
Note that $-\Gamma^{\prime\prime}(0)$ is the variance of $\dot{Z}(0)$. 

The jump epochs $\tau_n$ of $\Jc$ are such that for all $n$, 
$\tau_n-\tau_{n-1}$ are independent exponential random variables with intensity $\lambda<\infty$.
The process $\nu=(\nu_t:t\geq 0)$ with $\nu_t=\max\{n:\tau_n\leq t\}$ is called a Simple Poisson Process.

For the sake of notational simplicity we consider only the up-crossings, the case of down-crossings is 
completely analogous. 
\subsection{Stationary processes with $1$-dimensional Gaussian distribution}
Assume that $\Zc=(Z(t):t\in[0,T])$ is a centered stationary Gaussian process with $C^1$ paths. 
Assume also that $\Gamma(0)=1/2$.

Besides, let $\Jc=(J(t):t\in[0,T])$ be constructed in the following way.
Given a sequence of independent and identically distributed random variables $(\epsilon_n:n\in\N)$ with common centered Gaussian distribution with variance $1/2$, and $\rho\in\R$ such that $|\rho|<1$, define
\begin{equation*}
A_0=\epsilon_0,\textrm{ and for }n\geq 1: A_n=\rho A_{n-1}+\sqrt{1-\rho^2}\,\epsilon_{n}. 
\end{equation*}
Consider also a Simple Poisson Process $\nu=(\nu_t:t>0)$ with intensity $0<\lambda<\infty$. 
Assume that $\nu$ is independent from $(\epsilon_n)_n$. 
Finally, let
\begin{equation*}
J(t)=A_{\nu_{t}}. 
\end{equation*}

We call such a process a Poisson-auto-regressive process and denote it by PAR($\rho,\lambda$). 
Note that this process is similar, in certain sense, to the Random Telegraph Signal.

The following proposition is obtained by a direct computation conditioning 
on the number of jumps of $\Jc$. 

\begin{proposition}\label{prop:par}
Let $\Jc$ be a PAR($\rho,\lambda$) process, then $\Jc$ is wide-sense stationary 
and for $\tau\in[0,T]$ the r.v. $J(\tau)$ has centered Gaussian distribution with variance $1/2$ 
and the covariance between $J(0)$ and $J(\tau)$ is $e^{-\lambda(1-\rho)\tau}/2$.
\end{proposition}

The next theorem gives the mean number of continuous and discontinuous crossings through the level $u$ by 
the process $\X$ on the interval $[0,T]$. Let $\varphi$ denote the standard Gaussian density function.
Considering $(X,Y)$ a two dimensional centered Gaussian vector with $\Esp{X^2}=\Esp{Y^2}=1$ and $\Esp{XY}=(1+\rho)/2$, 
introduce
\begin{equation}\label{eq:perro}
P_{\rho}(u)=\Prob{(X<u,Y>u)},
\end{equation}
i.e. the probability that a two dimensional centered Gaussian vector with unit variances and covariance $(1+\rho)/2$ belongs to the set $(-\infty,u)\times(u,\infty)$.

\begin{theorem}\label{theorem:poiss-ar}
Let $\X=\Zc+\Jc$ with $\Zc,\Jc$ independent processes such that  
$\Zc$ is a centered  stationary Gaussian process with $\Gamma(0)=1/2$ and $C^1$ paths and $\Jc$ a PAR($\rho,\lambda$) process. 
Then 
\begin{equation*}
\Esp U^{c}_{u}=T\sqrt{\frac{-\Gamma^{\prime\prime}(0)}{2\pi}}\,\varphi(u),\quad
\Esp U^{d}_{u}=\lambda T\,P_{\rho}(u).
\end{equation*}
\end{theorem}
\begin{proof}
We begin considering the mean number of continuous crossings. Since $\Zc$ is Gaussian and stationary, 
we can apply Corollary \ref{theorem:gauusest}.
Besides, 
by Proposition \ref{prop:par}, the density $p_{X(t)}=p_{X{(0)}}=\varphi$ for all $t$. 
Therefore,
\begin{equation*}
\Esp U^{c}_{u}=T\,\Esp {\dot{Z}^{+}(0)}\,\varphi(u).
\end{equation*}
Routine computations shows that 
${\Esp{\dot{Z}}^{+}(0)=\sqrt{{-\Gamma^{\prime\prime}(0)}/{2\pi}}}$.
This gives $\Esp{U^c_u}$.

Let us consider now the discontinuous up-crossings. 
The compensating measure of the point process $(\tau_{n},\xi_{n})_{n}$, $\xi_n:=A_n-A_{n-1}$, 
is $\lambda dt\,F(dy)$, where $F$ 
is the normal distribution centered at $(\rho-1) A_{\nu_{t^{-}}}$ and with variance $(1-\rho^{2})/2$, 
see \cite[eq. 4.64]{Jac}. 
Hence
\begin{multline*}
\Esp{U^{d}_{u}}=\Esp\int^{T}_{0}\int_{\R}\indicator\{X(t^{-})<u,X(t^{-})+y>u\}\lambda dt\,F(dy)\\
=\lambda\Esp\int^{T}_{0}dt\int_{\R}\indicator\{X(t^{-})<u,X(t^{-})+y>u\}F(dy).
\end{multline*}
Actually $F$ is the distribution of $\xi_{\nu_t}$ conditioned on the random vector 
$(Z(t),A_{\nu_{t^{-}}})$. 
Thus
\begin{multline*}
\Esp{U^{d}_{u}}=\lambda\Esp\int^{T}_{0}\Prob\left(X(t^{-})<u,X(t^{-})+\xi_{\nu_t}>u\mid Z(t),A_{\nu_{t^{-}}}\right)dt\\
=\lambda\int^{T}_{0}\sum^{\infty}_{n=0}p_{\nu_{t}}(n)
\Esp\Prob\left(Z(t)+A_{n}<u,Z(t)+A_{n}+\xi_{n+1}>u\mid Z(t),A_{n}\right)dt\\
=\lambda\int^{T}_{0}\sum^{\infty}_{n=0}p_{\nu_{t}}(n)
\Prob\left(Z(t)+A_{n}<u,Z(t)+A_{n}+\xi_{n+1}>u\right)dt
\end{multline*}
where we have conditioned on the number of jumps and used Fubini's theorem.
Note that in this case a direct computation is possible. 
By the stationarity of $\Zc$ and $\Jc$, the last probability does not depend on $t$ and $n$. Hence
\begin{equation*}
\Esp{U^{d}_{u}}
=\lambda T\Prob\left(Z(0)+A_{0}<u,Z(0)+A_{0}+\xi_1>u\right)
=\lambda TP_{\rho}(u).
\end{equation*}
This gives $\Esp U^d_{u}$ and concludes the proof.
\end{proof}
Now we can see that the continuous crossings dominate, in mean, for high levels. 

Let $\Phi$ be the standard Gaussian distribution and $\overline{\Phi}(z)=1-\Phi(z)$. 
As usual, $f(u)=o(g(u))$ means that $\lim_{u\to\infty}f(u)/g(u)=0$.
\begin{corollary}\label{cor:maxpar}
As $u\to\infty$, we have $\Esp{U^{d}_{u}}=o\left(\Esp{U^{c}_{u}}\right)$. 
\end{corollary}
\begin{proof}
The results follows directly from \eqref{eq:perro}:
\begin{equation*}
P_{\rho}(u)\leq\Prob(Y>u)\leq 
\overline{\Phi}\left(u\right)=o\left(\varphi(u)\right).
\end{equation*}
\end{proof}

\begin{remark}
The processes $\Zc$ and $\Jc$ contribute in $1/2$ to the variance of $\X$. 
More generally we can define 
\begin{equation*}
\X=\alpha\Zc+\sqrt{1-\alpha^2}\Jc
\end{equation*}
for $\alpha\in[0,1]$ and $\Zc,\Jc$ independent, centered, stationary with variance one.
Note that $p_{X(t)}=\varphi$ for all $\alpha$ 
and $-\Gamma_{\alpha Z(0)}^{\prime\prime}(0)=-\alpha^2\Gamma_{Z(0)}^{\prime\prime}(0)$. Hence
\begin{equation*}
\Esp{N^{c}_{u}}(\X)
=\alpha\,\Esp{N^{c}_{u}}(\Zc).
\end{equation*}
Therefore, the mean number of continuous crossings of $\X$ is proportional to that of the continuous process $\Zc$, 
the constant of proportionality being $\alpha$, 
that is, the standard deviation of the first summand $\alpha\Zc$.
\end{remark}

\subsection{Stationary Gaussian continuous process plus CPP}
In this example we consider a centered stationary Gaussian process $\Zc$ with $\Gamma(0)=1$ and $C^1$ paths and 
an independent Compound Poisson Process (CPP), $\Jc$ with finite intensity $\lambda$ and standard Gaussian jumps. 

The process $\Jc$ is defined as follows: 
let $(\xi_{n})_{n\in\N}$ be independent standard Gaussian random variables 
and let $\nu=(\nu_t:t\geq 0)$ be a Simple Poisson Srocess  
independent from $(\xi_n)$. 
Then, for $t\geq 0$ define 
$J(t):=\sum^{\nu_{t}}_{n=1}\xi_{n}$. 

Thus, $J(t)$ is the sum of a random number of standard normal random variables. 
It is well known, see \cite{Jac} for instance, that the compensating measure of the CPP is $\lambda dt\,\Phi(dx)$. 
In particular, the compensating measure is deterministic.

The next theorem gives the mean number of up-crossings through $u$ by $\X$. 
Let $\varphi_n$ denote the centered Gaussian density with variance $n$, 
in particular, $\varphi_1=\varphi$. 
Note that $\varphi_n\ast \varphi_m=\varphi_{n+m}$, where $\ast$ stands for the convolution. 
Furthermore, let $p=\sum^{\infty}_{n=1}p_{n}\varphi_{n}$ with
$p_n=\frac{1}{\lambda T}\Prob(\nu_T\geq n)$ 
and note that the density function $p$ has expectation $0$ and variance $\lambda T/2$. 
\begin{theorem}\label{theoremcppcont}
For the process defined above we have
\begin{equation*}
\Esp{U^{c}_{u}}=T\,\sqrt{\frac{-\Gamma^{\prime\prime}(0)}{2\pi}}\,p(u),\quad
\Esp{U^{d}_{u}}=\lambda T\int^{u}_{-\infty} \overline{\Phi}(u-x)\,p(x)dx.
\end{equation*}
\end{theorem}
Observe that the mean number of continuous crossings has the usual form, see Remark \ref{remark:gauusest}.
\begin{proof}
We begin with the continuous crossings. 
Using the formula in Corollary \ref{cor:pzpj} we have
\begin{multline*}
\Esp{U^{c}_{u}}=\int^{T}_{0}dt\int^{\infty}_{-\infty}\Esp\left(\dot{Z}^{+}(t)\mid Z(t)=v\right)p_{Z(t)}(v)p_{J(t)}(u-v)dv\\
=\Esp\dot{Z}^{+}(0)\int^{T}_{0}dt\int^{\infty}_{-\infty}p_{Z(0)}(v)p_{J(t)}(u-v)dv,
\end{multline*}
where we used the stationarity of $\Zc$ as in Corollary \ref{theorem:gauusest}. 
Recall that $\Esp{\dot{Z}^{+}(0)}=\sqrt{-\Gamma^{\prime\prime}(0)/2\pi}$. 
Now, by Lemma \ref{lema} we decompose $p_{J(t)}$. Thus
\begin{multline*}
\Esp{U^{c}_{u}}
=\sqrt{\frac{-\Gamma^{\prime\prime}(0)}{2\pi}}\int^{T}_{0}dt\sum^{\infty}_{n=0}p_{\nu_t}(n)\int^{\infty}_{-\infty}p_{Z(0)}(v)\varphi_{n}(u-v)dv\\
=\sqrt{\frac{-\Gamma^{\prime\prime}(0)}{2\pi}}\int^{T}_{0}dt\sum^{\infty}_{n=0}p_{\nu_t}(n)(p_{Z(0)}\ast \varphi_{n})(u)\\
=\sqrt{\frac{-\Gamma^{\prime\prime}(0)}{2\pi}}\sum^{\infty}_{n=0}\varphi_{n+1}(u)\int^{T}_{0}p_{\nu_t}(n)dt.
\end{multline*}
This integral is computed in the second item of Lemma \ref{lema}. So we get
\begin{equation*}
\Esp{U^{c}_{u}}
=T\;\sqrt{\frac{-\Gamma^{\prime\prime}(0)}{2\pi}}\sum^{\infty}_{n=0}\frac{1}{\lambda T}\Prob\left(\nu_T\geq n+1\right)\varphi_{n+1}(u)
=T\;\sqrt{\frac{-\Gamma^{\prime\prime}(0)}{2\pi}}\;p(u).
\end{equation*}

Now we turn to the discontinuous crossings
\begin{multline*}
\Esp{U^{d}_{u}}=\Esp{\int^{T}_{0}\int^{\infty}_{-\infty}
\indicator{\{X(t^-)<u;X(t^-)+y>u\}}L(dt,dy)}\\
=\Esp{\int^{T}_{0}\int^{\infty}_{-\infty}
\indicator{\{X(t^-)<u;X(t^-)+y>u\}}\lambda dt\,\Phi(dy)}.
\end{multline*}
Since the compensating measure is deterministic we have
\begin{multline*}
\Esp{U^{d}_{u}}
=\lambda\int^{T}_{0}\int^{\infty}_{-\infty}
\Prob{\left(X(t^-)<u;X(t^-)+y>u\right)}dt\,\Phi(dy)\\
=\lambda\int^{T}_{0}\Prob{\left(X(t^-)<u;X(t^-)+\xi>u\right)}dt,
\end{multline*}
being $\xi$ a standard normal variable independent from $X(t^{-})$. 
Now we condition on $X(t^{-})$, since for fixed $t$, almost surely $J(t^{-})=J(t)$ we may use the same computations as in 
the computation of $\Esp{N^{c}_{u}}$. Hence
\begin{multline*}
\Esp{U^{d}_{u}}
=\lambda\int^{T}_{0}\int^{u}_{-\infty}\Prob{\left(\xi>u-x\mid X(t^-)=x\right)}\,p_{X(t^-)}(x)dxdt\\
=\lambda\int^{u}_{-\infty}\overline{\Phi}(u-x)dx\int^{T}_{0}p_{X(t^{-})}(x)dt\\
=\lambda T\int^{u}_{-\infty} \overline{\Phi}(u-x)p(x)dx.
\end{multline*}
This completes the proof.
\end{proof}

Next, we compare the mean numbers of continuous and discontinuous up-crossings through the level $u$ by $\X$ as $u\to\infty$.
\begin{corollary}
As  $u\to\infty$ the mean numbers of continuous and 
discontinuous up-crossings through the level $u$ are of the same order. 
More precisely,
\begin{equation*}
\lim_{u\to\infty} \frac{\lambda}{\sqrt{-\Gamma^{\prime\prime}(0)}}\,\Esp{U^{c}_{u}}
\leq\lim_{u\to\infty}\Esp{U^{d}_{u}}
\leq\lim_{u\to\infty} \lambda\,\sqrt{\frac{2\pi}{-\Gamma^{\prime\prime}(0)}}\,\Esp{U^{c}_{u}}
\end{equation*}
\end{corollary}
\begin{proof}
We bound $\Esp{U^{d}_{u}}$ from above and below.
First, we have
\begin{equation*}
\int^{u}_{-\infty}\overline{\Phi}(u-y)p(y)dy=
\int^{\infty}_{0}\overline{\Phi}(y)p(u-y)dy
\geq p(u)\int^{2u}_{0}\overline{\Phi}(y)dy,
\end{equation*}
where we used that $p$ is even and decreasing on $[0,\infty)$. 
Furthermore,
\begin{equation*}
\lim_{u\to\infty}\int^{2u}_{0}\overline{\Phi}(y)dy=
\int^{\infty}_{0}\overline{\Phi}(y)dy
=\Esp{\xi^{+}}=\sqrt{\frac{2}{\pi}}.
\end{equation*}
Therefore,
\begin{equation*}
\Esp{U^{d}_{u}}\geq \lambda T\,p(u)\int^{\infty}_{0}\overline{\Phi}(y)dy 
\mathop{\sim}\limits_{u\to\infty} \frac{\lambda T}{\sqrt{2\pi}} p(u).
\end{equation*}

On the other hand,
\begin{multline*}
\Esp{U^{d}_{u}}=\lambda T\sum^{\infty}_{n=1}p_{n}\int^{\infty}_{0}\overline{\Phi}(u-y)\varphi_{n}(y)dy
\leq \lambda T\sum^{\infty}_{n=1}p_{n}\int^{\infty}_{-\infty}\overline{\Phi}(u-y)\varphi_{n}(y)dy.
\end{multline*}
Note that this is a convolution formula for the tail of the 
distribution of two (independent) random variables, say $Z\sim\Phi$ and $V_{n}\sim \varphi_{n}$.  
Then the last integral equals $\Prob{(Z+V_n>u)}$. 
Furthermore, since $\Phi$ is the Gaussian standard distribution and $\varphi_{n}$ is the Gaussian density with zero mean and variance $n$, 
this probability equals $\Prob{(V_{n+1}>u)}=\overline{\Phi}(u/\sqrt{n+1})$. 
Thus
\begin{equation*}
\Esp{U^{d}_{u}}\leq \lambda T\sum^{\infty}_{n=1}p_{n}\overline{\Phi}\left(\frac{u}{\sqrt{n+1}}\right)
\leq \lambda T\sum^{\infty}_{n=1}(n+1)p_{n}\varphi_{n+1}(u)=\lambda T(p(u)+\delta(u)),
\end{equation*}
where $\delta(u):=\sum^{\infty}_{n=1}(n+1)p_{n}\varphi_{n+1}(u)-p(u)$. 
In order to obtain the desired result, it suffices to prove that $\delta(u)\to 0$ as $u\to\infty$. 
In fact,
\begin{equation*}
|\delta(u)|\leq\left|\sum^{\infty}_{n=1}\left[(n+1)p_{n}-p_{n+1}\right]\varphi_{n+1}(u)\right|+|p_{1}\varphi_{1}(u)|
\end{equation*}
It is clear that the second term of the r.h.s  tends to zero when $u\to\infty$. 
Let us look at the first term, 
note that $\sum^{\infty}_{n=1}(n+1)p_{n}$ is, roughly speaking, the second factorial moment of the Poisson distribution. 
In fact,
\begin{multline*}
\sum^{\infty}_{n=1}(n+1)p_{n}
=1+\frac{1}{\lambda T}\sum^{\infty}_{n=1}\sum^{\infty}_{k=n}np_{\nu_{T}}(k)
=1+\frac{1}{2\lambda T}\sum^{\infty}_{k=1}k(k-1)p_{\nu_{T}}(k)<\infty.
\end{multline*}
Thus, the first term is a mixture of Gaussian densities $\varphi_n$ (times a constant), 
hence it tends to zero.
Then $\delta(u)\to 0$ as $u\to\infty$ and
\begin{equation*}
\Esp{U^{d}_{u}}\leq\lambda T(p(u)+\delta(u))\mathop{\sim}\limits_{u\to\infty}\lambda Tp(u).
\end{equation*}
Now, the result follows putting together the two obtained bounds for $\Esp{U^{d}_{u}}$ and using Theorem \ref{theoremcppcont}.
\end{proof}

\begin{remark}
Note that in this case the discontinuous crossings through the level $u$ are not negligible w.r.t. 
the continuous crossings when $u$ tends to infinity, in contrast with the situation in the example of the previous section. 
\end{remark}
We end with an auxiliary lemma.
\begin{lemma}\label{lema}
Let $\Jc$ be the CPP with standard Gaussian jumps, then
\begin{enumerate}
\item the density of the CPP can be written as: 
$p_{J(t)}(x)=\sum^{\infty}_{n=1}p_{\nu_t}(n)\varphi_{n}(x)$.

\item $\int^{T}_{0}p_{\nu_t}(n)dt=\frac{1}{\lambda} \Prob(\nu_T\geq n+1)
=\frac{1}{\lambda}\Prob(\tau_{n+1}\leq T)$.

\item If $p_{n}=\frac{1}{\lambda T}\Prob(\nu_T\geq n)$, then $\sum^{\infty}_{n=1}p_{n}=1$
\end{enumerate}
\end{lemma}
\begin{proof}
1. Conditioning on the value of $\nu_{t}$ we have:
\begin{equation*}
F_{J(t)}(x)=\Prob(J(t)\leq x)=\sum^{\infty}_{n=0}p_{\nu_{t}}(n)\Prob(J(t)\leq x\mid\nu_{t}=n).
\end{equation*}
Now, if $\nu_{t}=n$, $J(t)$ is the sum of $n$ independent standard Gaussian random variables, hence, the conditional probability 
in the r.h.s. of the latter equation is the distribution of centered normal random variable with variance $n$. 
The result follows taking derivatives on both sides.

2. By definition, $p_{\nu_t}(n)$, as a function of $t$, 
is equal to the density function of the Gamma distribution with parameters $(\lambda,n+1)$ divided by $\lambda$,  
it is well known that this is the distribution of $\tau_{n+1}$, 
hence
\begin{equation*}
\int^{T}_{0}p_{\nu_t}(n)dt=\frac{\Prob(\tau_{n+1}\leq T)}{\lambda}.
\end{equation*}
The result follows since the events $\{\tau_{n+1}\leq T\}$ and $\{\nu_T\geq n+1\}$ coincide.

3. It follows directly from the facts that $\nu_{T}$ has Poisson distribution with mean $\lambda T$ 
and that for a non-negative integer valued random variable $X$: $\Esp{X}=\sum^{\infty}_{n=1}\Prob(X\geq n)$.
\end{proof}

\subsection{Application to the distribution of the maximum}
Recall that
\begin{equation*}
M(T)=\max\{X(s):0\leq s\leq T\}. 
\end{equation*}
We now use Rice Formula to get upper and lower bounds for 
$\Prob(M(T)>u)$ where $\X$ has differentiable stationary Gaussian continuous part 
and Poisson-auto-regressive jump part defined in Section 4.1. 

The following bounds are based on the elementary relation
\begin{equation*}
\{M(T)>u\}=\{X(0)>u\}\biguplus\,\{X(0)<u,U_{u}\geq 1\},
\end{equation*}
where $\uplus$ denotes the disjoint union. 
It follows that
\begin{equation}\label{markova}
\Prob(M(T)>u)\leq\Prob(X(0)>u)+\Prob(U_{u}\geq 1)
\leq\Prob(X(0)>u)+\Esp{U_{u}},
\end{equation}
and
\begin{multline}\label{markovb}
\Prob(M(T)>u)=\Prob(X(0)>u)+\Prob(U_u\geq 1)-\Prob(U_u\geq 1,X(0)>u)\\
\geq\Prob(X(0)>u)+\Esp{U_{u}}-\frac{1}{2}\Esp{U_{u.[2]}}-\Prob(U_u\geq 1,X(0)>u),
\end{multline}
where $a_{[2]}=a(a-1)$ is the Pochammer symbol and $U_{u.[2]}=(U_u)_{[2]}$

The following theorem contains the upper bound for the tail of the distribution of the maximum 
of $\X$ on the interval $[0,T]$.
\begin{theorem}\label{theorem:maxpar}
As $u\to\infty$, the tail of the 
distribution of the maximum verifies
\begin{equation}\label{eq:theoremmaxpar}
\Prob(M(T)>u)\leq 1-\Phi(u)+T\,\sqrt{\frac{{-\Gamma^{\prime\prime}(0)}}{{2\pi}}}\varphi(u)+
\lambda TP_{\rho}(u),
\end{equation}
where $P_{\rho}(u)$ is defined in \eqref{eq:perro}.
\end{theorem}
\begin{proof}
The proof is a direct consequence of Theorem \ref{theorem:poiss-ar} and formula \eqref{markova}.  
\end{proof}
\begin{remark}\label{remark:maxpar}
Furthermore, if we denote the r.h.s. of \eqref{eq:theoremmaxpar} by $rhs(u)$, 
by Corollary \ref{cor:maxpar} we have:
\begin{equation*}
rhs(u)\mathop{\sim}\limits_{u\to\infty}T\,\sqrt{\frac{{-\Gamma^{\prime\prime}(0)}}{{2\pi}}}\;\varphi(u).
\end{equation*}
\end{remark}

The goal of the rest of this section is to show that this upper bound is sharp. 
In order to do that, we use the lower bound given in equation \eqref{markovb} 
for the tail of the 
distribution of the maximum $M(T)$. 
So we have to deal with the second moment of the number up-crossings. 
\begin{theorem}
Let the processes $\Zc,\Jc$ and $\X$ be as in Corollary \ref{theorem:gauusest}. 
If in addition $\Jc$ is a PAR($\rho,\lambda$) process, 
and $\Zc$ verifies that $\Gamma(\tau)\neq\pm 1/2$ for all $\tau>0$ and the Geman condition:
\begin{equation}\label{geman}
\int^\eta_0\frac{\theta^{\prime}(\tau)}{\tau^{2}}dt\textrm{ converges for some }\eta>0,
\end{equation}
where $\theta$ is defined by $\Gamma(\tau)=1+\Gamma^{\prime\prime}(0)\tau^{2}/2+\theta(\tau)$. 
Then
\begin{equation*}
\Prob(M(t)>u)= 1-\Phi(u)+T\,\sqrt{\frac{-\Gamma^{\prime\prime}(0)}{2\pi}}\,\varphi(u)+o(\varphi(u))
\end{equation*}
\end{theorem}
\begin{proof}
It suffices to show that the additional terms in \eqref{markovb}, w.r.t. \eqref{markova}, are $o(\varphi(u))$. 
Since 
\begin{equation*}
U_{u,[2]}=U_{u}(U_{u}-1)=U^{c}_{u,[2]}+U^{d}_{u,[2]}+2\,U^{c}_{u}\,U^{d}_{u},
\end{equation*}
the proof is divided into several steps, considering separately each of the resulting terms.

\begin{claim} We have
$\Esp{U^{c}_{u,[2]}}=o(\varphi(u))$.
\end{claim}

{\bf Step 1.} 
We begin with an upper bound for the second moment of $U_{u}$. 
\begin{multline}\label{lema:2mcota}
\Esp{U^{c}_{u.[2]}}\\
\leq\int^{T}_{0}\int^{T}_{0}\sum^{\infty}_{m,n=0}
\Esp{\left[\dot{Z}^{+}(s)\dot{Z}^{+}(t)\mid X(s)=X(t)=u,\nu_{s}=m,\nu_{t-s}={n}\right]}\\ 
\cdot p_{X(s),X(t),\nu_{s},\nu_{t-s}}(u,u,m,n)dsdt.
\end{multline}
This bound is enough for our current purposes, 
but one can easily obtain the equality by the method of approximation by polygonals, see \cite{dalmao}.

We adapt the proof of the Rice formula for the factorial moments in \cite[Theorem 3.2]{aw}. 
Let $C_{u}$ be the set of continuous up-crossings of $\X$ in $[0,T]$, $C^{2}_{u}=C_{u}\times C_{u}$ and 
for any Borel set $J$ in $[0,T]^{2}$ let $\mu(J)=\#(C^{2}_{u}\cap J)$. 
It follows that $U^{c}_{u.[2]}=\mu([0,T]^{2}\setminus \Delta)$, where $\Delta$ is the diagonal, that is, 
$\Delta=\{(s,t)\in[0,T]^{2}:s=t\}$. 

Take $J_{1}$ and $J_{2}$ disjoint intervals in $[0,T]$ and let $J=J_{1}\times J_{2}(\subset[0,T]^{2}\setminus\Delta)$. 
Then
\begin{multline*}
\mu(J)=U_{u}(J_{1})\cdot U_{u}(J_{2})\\
=\lim_{\delta\to 0}\frac{1}{(2\delta)^{2}}\int_{J_{1}}\dot{Z}^{+}(s)\indicator\{|X(s)-u|<\delta\}ds
\cdot \int_{J_{2}}\dot{Z}^{+}(t)\indicator\{|X(t)-u|<\delta\}dt\\
=\lim_{\delta\to 0}\frac{1}{(2\delta)^{2}}\int\int_{J}\dot{Z}^{+}(s)\dot{Z}^{+}(t)\indicator\{|X(s)-u|<\delta,|X(t)-u|<\delta\}dsdt,
\end{multline*}
where we applied Kac Formula on each interval, and noted that for $\delta$ small enough 
the quantity in the limit becomes constant, so we can use the same mute variable $\delta$ in both limits.

Now, we take expectation on both sides and apply Fatou's Lemma and Fubini's Theorem 
to pass the expectation inside the integral sign. 
Then, we condition on the number of jumps of the process in the intervals $[0,s],[0,t]$ and on the values of the process $\X$
at these points, that is (if $s<t$):
\begin{multline*}
\Esp{\dot{Z}^{+}(s)\dot{Z}^{+}(t)\indicator\{|X(s)-u|<\delta,|X(t)-u|<\delta\}}=\\
\int^{u+\delta}_{u-\delta}\int^{u+\delta}_{u-\delta}\sum^{\infty}_{m,n=0}
\Esp\left[\dot{Z}^{+}(s)\dot{Z}^{+}(t)\mid X(s)=x,X(t)=y,\nu_{s}=m,\nu_{t}=m+n\right]\\
p_{X(s),X(t)\mid\nu_{s}=m,\nu_{t}-\nu_{s}=n}(x,y)p_{\nu_s}(m)p_{\nu_t}(m+n)dxdy
\end{multline*}
Observe that, under these conditions $X(s)=Z(s)+A_{m}$ and $X(t)=Z(t)+A_{m+n}$ 
($A_m$ is the basic auto-regressive sequence defined at the beginning of Section 4.1) are jointly Gaussian, 
thus the conditional expectation is well defined (and may be computed) via regression. 
Besides, this fact yields the regularity conditions needed for the integrand. 
In fact, the conditional expectation and the joint density function are continuous 
for $t\in[0,T]$ and $x,y$ in a neighborhood of $u$. 
Hence, we can pass to the limit w.r.t. $u$ inside the integral sign. 

So far we have stated that $\Esp{\mu(J)}$ is bounded by the integral in the r.h.s. of \eqref{lema:2mcota} for any 
interval $J\subset[0,T]^{2}\setminus\Delta$.
As both sides represent measures on $[0,T]^{2}\setminus\Delta$, 
the result follows by the standard arguments of Measure Theory. 

{\bf Step 2.} Actually we can take inequality \eqref{lema:2mcota} a little further:
\begin{multline*}
\Esp{U^{c}_{u.[2]}}\leq 2\int^{T}_{0}(T-\tau)\sum^{\infty}_{n=0}
\Esp{\left[\dot{Z}^{+}_{0}\dot{Z}^{+}_{\tau}\mid Y_{0}(0)=Y_n(\tau)=u\right]}\cdot\\ 
\cdot p_{Y_{0}(0),Y_{n}(\tau)}(u,u)p_{\nu_{\tau}}(n)d\tau,
\end{multline*}
where we set $Y_{k}(t)=Z(t)+A_{k}$, for $t\in[0,T]$ and $k\in\N$.

In fact, conditioned on $\nu_s=m,\nu_{t-s}=n$ (if $s<t$; similarly on the other case) 
we have $X(s)=Z(s)+A_{m}=:Y_{m}(s)$ and $X(t)=Z(t)+A_{m+n}=:Y_{m+n}(t)$.
It is easy to see that the vector $(Y_{m}(s),Y_{m+n}(t))$ is independent from $\nu$ and 
has centered normal distribution with variances $1$ and covariance $\Gamma_{t-s}+\rho^{n}/2$, 
in particular, this law does not depend on $s,t$ but on the difference $t-s$, 
neither does it depend on $m$. 
Therefore, the conditional expectation in inequality \eqref{lema:2mcota} reduces to that in the r.h.s. of the claimed bound. 

Then, we factorize the density function as:
\begin{multline*}
p_{X(s),X(t),\nu_{s},\nu_{t-s}}(u,u,m,n)=p_{X(s),X(t)\mid\nu_{s},\nu_{t-s}}(u,u)p_{\nu_s}(m)p_{\nu_{t-s}}(n)\\
=p_{Y_{m}(s),Y_{m+n}(t)}(u,u)p_{\nu_s}(m)p_{\nu_{t-s}}(n)\\
=p_{Y_{0}(0),Y_{n}(\tau)}(u,u)p_{\nu_s}(m)p_{\nu_{\tau}}(n).
\end{multline*}
Clearly $\sum_{m}p_{\nu_s}(m)=1$. 
Finally we make the change of variables $(s,t)\mapsto(s,\tau=t-s)$, and 
obtain the desired inequality. 

In the next two steps we bound each factor in the integrand.

{\bf Step 3.} 
Ordinary computations, similar to those in Proposition 4.2 of \cite{aw} 
show that 
\begin{equation*}
\Esp(\dot{Z}^{+}(0)\dot{Z}^{+}(\tau)\mid X_{0}(0)=X_{n}(\tau)=u)\leq
-\Gamma^{\prime\prime}(0)-\frac{{\Gamma^{\,\prime}(\tau)}^{2}}{1-(\Gamma(\tau)+\rho^{n}/2)^{2}}.
\end{equation*}
Note that for $n\geq 1$, as $\rho< 1$, there is no problem when $\tau\to 0$. 
In fact, $1-(\Gamma(\tau)+\rho^{n}/2)\to 1/2-\rho^{n}/2>0$. 

{\bf Step 4.} The vector $(Y_{0}(0),Y_{n}(\tau))$ is normally distributed with variances $1$ and covariance $\Gamma(\tau)+\rho^{n}/2$, 
therefore, the exponential in the density is 
\begin{equation*}
\exp{\left\{-\frac{u^{2}}{1+\Gamma(\tau)+\rho^{n}/2}\right\}}.
\end{equation*}
The case $n=0$, when there are no jumps in $[0,\tau]$, is treated as in \cite[Proposition 4.2]{aw}, 
in particular, we need Geman condition to ensure the convergence of the integral.

For $n\geq 1$ we can bound $\rho^{n}\leq|\rho|<1$. 
Hence
\begin{equation*}
\exp{\left\{-\frac{u^{2}}{1+\Gamma(\tau)+\rho^{n}/2}\right\}}
<\exp{\left\{-\frac{u^{2}}{3/2+\Gamma(\tau)}\right\}}
=o(\varphi(u)).
\end{equation*}
Then, replacing this in inequality \eqref{lema:2mcota} of Step 1, we have 
$
\Esp{U^{c}_{u.[2]}}
=o(\varphi(u))
$.
 
\begin{claim} We have
$\Esp{U^{d}_{u,[2]}}=o(\varphi(u))$.
\end{claim}
By the arguments in Corollary \ref{cor:maxpar} it suffices to show that 
$\Esp{U^{d}_{u,[2]}}\leq c\,P_{|\rho|}(u)$, 
for some constant $c$ and $u$ large enough. 
Recall that $P_{|\rho|}(u)$ is defined in \eqref{eq:perro}. 
We can write
\begin{equation*}
U^{d}_{u}=\sum^{\nu_T}_{n=0}\sum^{n}_{k=0}\indicator\{X(\tau^{-}_{k})<u,X(\tau_{k})>u\}.
\end{equation*}
Hence, making the product and taking expectation, we have:
\begin{multline}\label{suma}
\Esp{U^{d}_{u,[2]}}=\sum^{\infty}_{n=2}\sum^{n-1}_{1=k<\ell}p_{\nu_{T}}(n)\\
\cdot\Prob(Z(\tau_{k})+A_{k-1}<u;Z(\tau_{k})+A_{k}>u;Z(\tau_{\ell})+A_{\ell-1}<u;Z(\tau_{\ell})+A_{\ell}>u)\\
\leq \sum^{\infty}_{n=2}\sum^{n-1}_{k<\ell=1}p_{\nu_{T}}(n)
\Prob(Z(\tau_{k})+A_{k-1}<u;Z(\tau_{\ell})+A_{\ell}>u).
\end{multline}
Besides,
\begin{multline*}
\Prob(Z(\tau_{k})+A_{k-1}<u;Z(\tau_{\ell})+A_{\ell}>u)\\
=\int^{T}_{0}ds\int^{T}_{s}dtp_{\tau_{k},\tau_{\ell}\mid\nu_{T}=n}(s,t)
\Prob(Z(s)+A_{k-1}<u;Z(t)+A_{\ell}>u)\\
=\int^{T}_{0}ds\int^{T}_{s}dtp_{\tau_{k},\tau_{\ell}\mid\nu_{T}=n}(s,t)
\Prob(Z(0)+A_{k-1}<u;Z(t-s)+A_{\ell}>u),
\end{multline*}
in the last equality we used the stationarity of the process $\Zc$. 
The vector $(Z(0)+A_{k-1};Z(t-s)+A_{\ell})$ is centered Gaussian with variances $1$ and covariance $\Gamma(\tau)+\rho^{\ell-k+1}/2$. 
It is easy to check that $-\Gamma(\tau)-\rho^{\ell-k+1}/2<1/2(1+|\rho|)$ which 
is the covariance of the vectors $Z+S$ and $Z+|\rho| S+\sqrt{1-\rho^2}V$. 
Therefore, by the Plackett-Slepian inequality, see \cite[Section 2.1]{aw} we have
\begin{equation*}
\Prob(Z(0)+A_{k-1}<u;Z(t-s)+A_{\ell}>u)\leq\Prob(Z+S<u;Z+|\rho| S+\sqrt{1-\rho^2}V>u).
\end{equation*}
This bound does not depend on $s,t$. 
Hence
\begin{multline*}
\Prob(Z(\tau_{k})+A_{k-1}<u;Z(\tau_{\ell})+A_{\ell}>u)\\
\leq \Prob(Z+S<u;Z+|\rho| S+\sqrt{1-\rho^2}V>u)\int^{T}_{0}ds\int^{T}_{s}dtp_{\tau_{k},\tau_{\ell}\mid\nu_{T}=n}(s,t)\\
=\Prob(Z+S<u;Z+|\rho| S+\sqrt{1-\rho^2}V>u)\\
=P_{|\rho|}(u),
\end{multline*}
since $\tau_{k},\tau_{\ell}\mid\nu_{T}=n$ is concentrated on $[0,T]^2$.
Finally, replacing in the equation \eqref{suma} we have
\begin{equation*}
\Esp{U^{d}_{u,[2]}}
\leq \sum^{\infty}_{n=2}\sum^{n-1}_{1=k<\ell}P_{|\rho|}(u)
=\frac{(\lambda T)^{2}}{2}P_{|\rho|}(u),
\end{equation*}
and the result follows. 

\begin{claim}
We have $\Esp{U_{u,[2]}}=o(\varphi(u))$.
\end{claim}
In fact
\begin{multline*}
\Esp{U_{u,[2]}}=\Esp{(U^{c}_{u}+U^{d}_{u})(U^{c}_{u}+U^{d}_{u}-1)}\\
=\Esp{U^{c}_{u,[2]}}+\Esp{U^{d}_{u,[2]}}+2\Esp{U^{c}_{u}U^{d}_{u}}\\
\leq\Esp{U^{c}_{u,[2]}}+\Esp{U^{d}_{u,[2]}}+2\sqrt{\Esp{(U^{c}_{u})^{2}}\Esp{(U^{d}_{u})^{2}}}\\
=\Esp{U^{c}_{u,[2]}}+\Esp{U^{d}_{u,[2]}}
+2\sqrt{[\Esp(U^{c}_{u,[2]})+\Esp(U^{c}_{u})][\Esp(U^{d}_{u,[2]})+\Esp(U^{d}_{u})]},
\end{multline*}
where we used Cauchy-Schwarz inequality.
The first two terms in the r.h.s. are treated in the previous lemmas, 
under the square root sign the first factor is equivalent to $\varphi(u)$ and the second one is $o(\varphi(u))$. 
Thus, $\Esp{U_{u,[2]}}=o(\varphi(u))$.

\begin{claim} We have
$\Prob(X(0)>u,U_{u}\geq 1)=o(\varphi(u))$. 
\end{claim}
We follow the proof of the analogue assertion in \cite[Proposition 4.2]{aw}. 
The key fact is that the distribution of the process $\Jc$ remains unchanged under time reversal $t\mapsto T-t$. 

In fact, let us condition on the number of jumps $\nu_T=n$. 
Then, it is easy to check that the (conditional) distribution of $(A_1,A_2,\dots,A_n)\mid\nu_{T}=n$ is the same as the 
distribution of $(A_n,A_{n-1},\dots,A_1)\mid\nu_{T}=n$. 
Besides, the distribution of $\tau_1,\tau_2,\dots,\tau_n\mid\nu_T=n$ is that of a uniform (ordered) sample of size $n$, so it looks the same from $0$ and from $T$. 
Since the construction of the process $\Jc$ depends on these elements and there is no difference if we start at $0$ or at $T$ the claim follows.

In conclusion,
\begin{multline*}
\Prob(M(t)>u)\geq 1-\Phi(u)+T\,\sqrt{\frac{-\Gamma^{\prime\prime}(0)}{2\pi}}\,\varphi(u)
+\lambda T\,P_{\rho}(u)+O(\varphi((1+\delta)u))\\
=1-\Phi(u)+T\,\sqrt{\frac{-\Gamma^{\prime\prime}(0)}{2\pi}}\,\varphi(u)+o(\varphi(u)).
\end{multline*}
Taking into account
\eqref{eq:theoremmaxpar}, this completes the proof.
\end{proof}

\section{Generalization of Borovkov-Last's formula}
Borovkov and Last in \cite{bor-last} are interested in the continuous crossings 
through a level $u$ by a stationary Piecewise Deterministic Markov Process. 
A process $\X$ of this kind, starts at a random position, then jumps a random quantity at random times 
but moves deterministically between jumps. 

Such a process is described by a general point process $(\tau_n,\xi_n)_n$, 
and a (non-random) continuous rate function $r:\R\to\R$. 
More precisely, the process $\X$ has jumps at the points $(\tau_n)$,  
the magnitude of the jump is $\xi_n$ and 
on the interval $[\tau_n,\tau_{n+1})$, 
$X(t)$ follows the integral curve of $r$ (that is $\dot{X}(t)=r(X(t))$) 
with initial condition $X(\tau_{n})=X(\tau^{-}_{n})+\xi_n$.

Note that the jump part of the process is not independent from the continuous one. 
Actually, observe that if $r(u)>0$ (resp. $<$) the continuous crossings through the level $u$ can only be up-crossings (resp. down-crossings). 
Hence, between two consecutive discontinuous crossings there can be only one continuous crossing.
Therefore, the number of continuous crossings is obtained from the 
discontinuous up and down-crossings.

The next theorem extends Borovkov-Last Formula to the non-stationary case. 
Let $D_{r}=\{u:r(u)=0\}$. 
\begin{theorem}\label{theorem:bl}
Let $u\not\in D_{r}$ and assume that 
$r$ and $p_{X(t)}$ are continuous w.r.t. $x$ in a neighborhood of $u$ and $t\in[0,T]$.
Then:
\begin{equation*}\label{blrice}
\Esp N^{c}_{u}=|r(u)|\int^{T}_{0}p_{X(t)}(u)dt
\end{equation*}
\end{theorem}
\begin{proof}
For the levels $u\not\in D_{r}$ we can apply Kac counting formula pathwise 
for almost all paths of $\X$. 
In fact, the continuity of $r$ implies that the paths are of class $C^{1}$ between the jumps 
and that $\dot{X}(t)=r(X(t))$ for almost all $t\in[0,T]$. 
Since $X(t)$ has a density, 
the value $u$ is not taken at the extremes of the interval neither at the jump points almost surely. 
Furthermore, by the continuity of this density, there are not tangencies at level $u$. 

Now we take expectation on both sides of Kac Counting Formula and observe 
that the number of continuous crossings of the level $u\not\in D_{r}$ is bounded by the number of jumps $+1$ of $\X$ in $[0,T]$. 
In fact, the sign of $r(u)$ determines the direction of the continuous crossings of $u$, so, 
between two continuous crossings there must be a discontinuous one in the opposite direction, 
thus there is at most one continuous crossing at each of the intervals of the partition $\tau_0,\dots,\tau_{\nu_T},T$.
Then, since $\nu_T$ in integrable, we may pass to the limit under the expectation sign:
\begin{equation*}
\Esp{N^{c}_{u}}=\lim_{\delta\downarrow 0}\frac{1}{2\delta}
\int^T_0\Esp\left[|\dot{X}(t)|\indicator_{\{|X(t)-u|<\delta\}}\right]dt.
\end{equation*}
Now, $\dot{X}(t)$ is a deterministic function of $X(t)$, namely $\dot{X}(t)=r(X(t))$, so 
the integrand is simply the expectation of a function of $X(t)$. 
Therefore,
\begin{equation*}
\Esp{N^{c}_{u}}=\lim_{\delta\downarrow 0}\frac{1}{2\delta}
\int^T_0\int^{u+\delta}_{u-\delta}|r(x)|p_{X(t)}(x)dxdt.
\end{equation*}
By the continuity of the integrand and the compactness of the domain we can pass the limit inside the integral w.r.t. $t$. 
Then, the result follows by the mean value theorem.
\end{proof}
As a corollary, when $\X$ is stationary we obtain Borovkov-Last's Formula.
\begin{corollary}
If in addition to the conditions of Theorem \ref{theorem:bl} 
the process $\X$ is stationary, then
\begin{equation*}
\Esp N^{c}_{u}=|r(u)|p_{X(0)}.
\end{equation*}
\end{corollary}

\section{Proof of Theorem \ref{teo:nogauss}}
We need the following relations.
Let $T,X,Y,Z$ be random variables. Then
\begin{multline*}
\Esp(T\mid X=x,Y=y)=\int^{\infty}_{-\infty}\Esp(T\mid X=x,Y=y,Z=z)p_{Z\mid X=x,Y=y}(z)dz.
\end{multline*}
\begin{multline*}
\Esp(T\mid X=x,Z=z)p_{X\mid Z=Z}(x)\\=\int^{\infty}_{-\infty}\Esp(T\mid X=x,Y=y,Z=z)p_{X,Y\mid Z=z}(x,y)dy. 
\end{multline*}

\begin{proof}[Proof of Theorem \ref{teo:nogauss}]
Let us start with the formula for the mean number of continuous crossings, 
We condition on the number of jumps, $\nu_T=n$ and on the jump epochs, $\tau_k=t_k$. Thus
\begin{multline*}
\Esp{N^{c}_{u}}=
\Esp\left(\Esp\left[N^{c}_{u}\left(\X,[0,T]\right)\mid\nu_T=n;\boldsymbol{\tau}={\bf t}\right]\right)\\
=\sum^{\infty}_{n=0}p_{\nu_T}(n)\int_{[0,T]^{n}}p_{\boldsymbol{\tau}}({\bf t})\;
\Esp\left[N^{c}_{u}\left(\X,[0,T]\right)\mid\nu_T=n;\boldsymbol{\tau}={\bf t}\right]d{\bf t},
\end{multline*}
where we set ${\boldsymbol{\tau}}=(\tau_1,\dots,\tau_n)$ and ${\bf t}=(t_1,\dots,t_n)$. 

Now we look at the integrand and since the number of crossings is additive w.r.t. the interval, 
we split the interval $[0,T]$ as the union of the intervals $\Ic_{k}:=[t_{k-1},t_{k})$. 
Then
\begin{equation*}
\Esp\left[N^{c}_{u}\left(\X,[0,T]\right)\mid\nu_T=n;\boldsymbol{\tau}={\bf t}\right]
=\sum^{n}_{k=1}\Esp\left[N^{c}_{u}\left(\X,\Ic_{k}\right)\mid\nu_T=n;\boldsymbol{\tau}={\bf t}\right].
\end{equation*}
Each term can be written as
\begin{equation*}
\int^{\infty}_{-\infty}\Esp\left[N^{c}_{u}\left(\X,\Ic_{k}\right)\mid\nu_T=n;\boldsymbol{\tau}={\bf t};J(t_{k-1})=y\right]
p_{J(t_{k-1})\mid\nu_T=n;\boldsymbol{\tau}={\bf t}}(y)dy.
\end{equation*}
Now, conditionally on $\nu_T=n;\boldsymbol{\tau}={\bf t}$ and $J(\tau_{k-1})=y$, the process $\X$ can be written as $\Zc+y$ on $\Ic_{k}$. 
Since $\Zc$ verifies the conditions $A1$, $A2$, $A3$ and $A4$ on $\Ic_{k}$ so does the process $\Zc+y$. Therefore, 
we may apply Rice Formula, see \cite{aw}, on each interval under these conditions to obtain
\begin{multline*}
\Esp\left[N^{c}_{u}\left(\X,\Ic_{k}\right)\mid\nu_T=n;\boldsymbol{\tau}={\bf t};J(t_{k-1})=y\right]\\
=\int^{\tau_k}_{\tau_{k-1}}\Esp\left[|\dot{Z}(t)|\mid X(t)=u,\nu_T=n;\boldsymbol{\tau}={\bf t},J(\tau_{k-1})=y\right]\\
\cdot p_{X(t)\mid \nu_T=n;\boldsymbol{\tau}={\bf t},J(\tau_{k-1})=y}(u)dt.
\end{multline*}

Finally, we have to integrate (the conditions), for notational simplicity, let us write 
$g(n,{\bf t},y)=\Esp\left[|\dot{Z}(t)|\mid X(t)=u,\nu_{T}=n;\boldsymbol{\tau}={\bf t},J(\tau_{k-1})=y\right]$, 
$g(n,{\bf t})=\Esp\left[|\dot{Z}(t)|\mid X(t)=u,\nu_T=n;\boldsymbol{\tau}={\bf t}\right]$ 
and $g(n)=\Esp\left[|\dot{Z}(t)|\mid X(t)=u,\nu_T=n\right]$. 
Let us perform the integrals one by one, starting w.r.t. $y$, (use Fubini). 
Then
\begin{multline*}
\int_{\R}g(n,{\bf t},y)p_{X(t)\mid \nu_T=n;\boldsymbol{\tau}={\bf t},J(\tau_k)=y}(u)p_{J(t_k)\mid\nu_T=n;\boldsymbol{\tau}={\bf t}}(y)dy\\
=\int_{\R}g(n,{\bf t},y)p_{(X(t),J(\tau_k))\mid \nu_T=n;\boldsymbol{\tau}={\bf t}}(u,y)dy\\
=g(n,{\bf t})p_{X(t)\mid\nu_T=n;\boldsymbol{\tau}={\bf t}}(u).
\end{multline*}
Now, we sum the integrals over $k$ and integrate w.r.t. ${\bf t}$: 
\begin{multline*}
\int^{T}_{0}dt\int_{[0,T]^{n}}g(n,{\bf t})p_{X(t)\mid\nu_T=n;\boldsymbol{\tau}={\bf t}}(u)p_{\boldsymbol{\tau}\mid\nu_{T}=n}({\bf t})d{\bf t}\\
=\int^{T}_{0}g(n)p_{X(t)\mid\nu_{T}=n}(u)dt.
\end{multline*}
By the same arguments one can remove the condition on $\nu_{T}$. 
The result follows.

Now we proceed to the formula for the mean number of discontinuous crossings through level $u$. 
For the definitions used below see \cite{Jac}. 

Clearly, $\X$ only can have a discontinuous crossing through $u$ at the points $\tau_n;n=1,\dots,\nu_{T}$, 
the jump of $\X$ at each of these points is due to the jump of $\Jc$.  
Hence, we consider the Marked Point Process $((\tau_{k},\xi_{k}):k\geq 0)$ associated to $\Jc$ on $[0,\infty)\times\R$, 
which defines a Random Counting Measure $\mu(dt,dy)$, in terms of which 
we can write:
\begin{multline*}
N^{d}_{u}=\sum^{\nu_{T}}_{k=1}\indicator\{(X(\tau^{-}_{k})-u)(X(\tau^{-}_{k})+\Delta X(\tau_{k})-u)<0\}\\
=\sum_{0\leq t\leq T}\indicator\{(X(t^{-})-u)(X(t^{-})+\xi_{\nu_{t}}-u)<0\}\\
=\int_{[0,T]\times\R}\indicator\{(X(t^{-})-u)(X(t^{-})+y-u)<0\}\mu(dt,dy).
\end{multline*}
It is easy to see that this RCM has a compensating measure denoted by $L(dt,dy)$, see \cite{Jac} again. 
Taking expectations on both sides we have:
\begin{multline*}
\Esp{N^{d}_{u}}=\Esp\left[\Esp(N^{d}_{u}\mid \Zc)\right]\\
=\Esp\left[\Esp\left[\int_{[0,T]\times\R}\indicator\{(X(t^{-})-u)(X(t^{-})+y-u)<0\}\mu(dt,dy)\mid \Zc\right]\right]\\
=\Esp\int_{[0,T]\times\R}\indicator\{(X(t^{-})-u)(X(t^{-})+y-u)<0\}L(dt,dy)
\end{multline*}
where we used that conditioned on $\Zc$ the integral is done w.r.t. the RCM $\mu$ associated with 
$\Jc$ 
which coincides with the integral w.r.t. the compensating measure $L(dt,dy)$ since 
the integrand is predictable, in fact it is a function of $t^{-}$.
Finally we integrate with respect to $\Zc$ and the result follows.
\end{proof}


\begin{thebibliography}{10}
\bibitem{adler-samo}
Adler, R. and Samorodnitsky, G. (1997). Level crossings of absolutely continuous
stationary symmetric $\alpha$-stable processes. Ann. Appl. Probab. 7 460-493.

\bibitem{hyper}
Alodat M.T.,Aluudat K.M.  (2008) 
The generalized hyperbolic process. 
Brazilian Journal of Probability and Statistics. 
Volume 22, Number 1, pages 1-8.

\bibitem{arw}
Armentano, D; Wschebor, M. (2009) 
Random systems of polynomial equations. The expected number of roots under smooth analysis. 
Bernoulli 15, no. 1, 249-266. 

\bibitem{aw}
Aza\"{i}s, J.M.; Wschebor, M. (2009) 
Level sets and extrema of random processes and fields. 
John Wiley and Sons, Inc., Hoboken, NJ. xii+393 pp. ISBN: 978-0-470-40933-6 

\bibitem{sea waves}
Aza\"{i}s, J.M.; Le\'{o}n, J.R.; Wschebor, M. (2011) 
Rice formulas and Gaussian waves.
Bernoulli 17, no. 1, 170--193. 

\bibitem{shot} 
Bierm\'{e}, H.; Desolneux, A. (2012) 
Crossings of smooth shot noise processes. To appear in Ann. Appl. Prob.

\bibitem{shot-saltos}
Bierm\'e, H. and Desolneux, A. (2012) A Fourier approach for the crossings of Shot
Noise processes with jumps {\sl J. Appl. Probab.}, 49, pp. 100--113.

\bibitem{bor-last}
Borovkov, K.; Last, G. (2008) 
On level crossings for a general class of piecewise-deterministic Markov processes. 
Adv. in Appl. Probab. 40, no. 3, 815--834. 

\bibitem{bor-lastn}
Borovkov, K.; Last, G. (2012) 
On Rice's Formula for stationary multivariate piecewise deterministic smooth processes. 
J. Appl. Probab., 49(2):351 pp.363.


\bibitem{brillinger}
Brillinger, D. R. (1972). On the number of solutions of systems of random equations.
Ann. Math. Statist. 43, pp. 534-540.

\bibitem{Mario counting}
Cucker, F.; Krick, T.; Malajovich, G.; Wschebor, M. (2012) 
A numerical algorithm for zero counting. III: Randomization and condition. 
Adv. in Appl. Math. 48, no. 1, 215--248. 

\bibitem{dalmao}
Dalmao, F. (2013)
Rice formula: extensions and applications.
PhD Thesis. Pedeciba-Universidad de la Rep\'ublica, Uruguay.

\bibitem{galtier}
Galtier, T.  (2011)  Note on the Estimation of Crossing Intensity for Laplace Moving
Average. Extremes, 14, pp. 157-166.

\bibitem{Jac}
Jacobsen, M (2006) 
Point process theory and applications. 
Marked point and piecewise deterministic processes. 
Probability and its Applications. 
Birkh\"{a}user Boston, Inc., Boston, MA,. xii+328 pp. ISBN: 978-0-8176-4215-0; 0-8176-4215-3 

\bibitem{kac}
Kac, M.,(1943)  On the average number of roots of a random algebraic equation, Bull.
Amer. Math. 49, pp. 314-320

\bibitem{kratz}
Kratz, M. F. (2006) 
Level crossings and other level functionals of stationary Gaussian processes. 
Probab. Surv. 3, 230--288. 

\bibitem{kree}
Kr\'{e}e, P.; Soize, C. (1983) 
M\'{e}canique al\'{e}atoire. Vibrations non lin\'{e}aires, turbulences, s\'{e}ismes, houle, fatigue. 
Dunod, Paris. xv+644 pp. ISBN: 2-04-015501-5 

\bibitem{Lead} 
Leadbetter, M. R.; Spaniolo, G. V. (2004) 
Reflections on Rice's formulas for level crossings-history, extensions and use. 
Festschrift in honour of Daryl Daley. Aust. N. Z. J. Stat. 46, no. 1, 173-180

\bibitem{LLR}
Leadbetter, M. R.; Lindgren G.; Rootz\'{e}n H, (1983) 
Extremes and related properties of
stationary sequences and processes. Springer-Verlag, New York, Heidelberg, Berlin.

\bibitem{longuett} 
Longuett-Higgins M.S. (1957) 
The statistical analysisof a random moving surface. 
Philos. Trans. Roy. Soc. London, Ser. A, 249, 321-387.

\bibitem{marcus}
Marcus  M.B. (1977) Level crossings of a stochastic process with absolutely
continuous sample paths, The Annals of Probability, 5, pp. 52-71.

\bibitem{Micro}
Petters, A. O.; Rider, B.; Teguia, A.M. (2009) 
A mathematical theory of stochastic microlensing. II. Random images, shear, and the Kac-Rice formula. 
J. Math. Phys. 50, no. 12, 122501, 17 pp. 

\bibitem{Rice44}
Rice, S. O.  (1944) 
Mathematical analysis of random noise. 
Bell System Tech. J. 23, 282--332. 

\bibitem{Rice45}
Rice, S. O. (1945) 
Mathematical analysis of random noise. 
Bell System Tech. J. 24. 46--156. 

\bibitem{Rychlik}
Rychlik, I. (2001) 
On some reliability applications of Rice's formula for the intensity of level crossings. 
Extremes 3, no. 4, 331--348. 

\bibitem{scheutzow}
Scheutzow M. (1994) A law of large numbers for upcrossing measures, {\sl Stochastic
Processes and their Applications}, {\bf 53}, pp. 285-305.

\bibitem{zahle}
Z\"ahle, U.   (1984) A general Rice formula, Palm measures, and horizontal window
conditioning for random fields, {\sl Stochastic Process and Their Applications}, 17,
pp. 265-283.
\end{thebibliography}
\end{document}